\documentclass[journal]{IEEEtran}
\usepackage{amsthm}
\newtheorem{theorem}{Theorem}
\newtheorem{lemma}{Lemma}
\theoremstyle{definition}
\newtheorem{definition}{Definition}

\newtheorem{remark}{Remark}

\usepackage{subfigure}
\usepackage{float}
\usepackage{amsfonts}
\usepackage{verbatim}
\usepackage{amssymb}

\usepackage{algorithm}
\usepackage{algorithmic}

\ifCLASSINFOpdf
\else
   \usepackage[dvips]{graphicx}
   \graphicspath{{../eps/}}
   \DeclareGraphicsExtensions{.eps}
\fi
\ifCLASSOPTIONcompsoc
 \usepackage[caption=false,font=normalsize,labelfont=sf,textfont=sf]{subfig}
\else
 \usepackage[caption=false,font=footnotesize]{subfig}
\fi
\begin{document}
%
\title{Nonconvex fraction function recovery sparse signal by convex optimization algorithm}
%
%
%

\author{Angang~Cui,
        Jigen~Peng,
        Haiyang~Li
        and~Meng Wen
\thanks{A. Cui is with the School of Mathematics and Statistics, Xi'an Jiaotong University, Xi'an, 710049, China. e-mail: (cuiangang@163.com).}
\thanks{J. Peng and H. Li are with the School of Mathematics and Information Science, Guangzhou University, Guangzhou, 510006, China. e-mail: (jgpengxjtu@126.com; fplihaiyang@126.com).}
\thanks{M. Wen is with the School of Science, Xi'an Polytechnic University, Xi'an, 710048, China. e-mail: (wen5495688@163.com).}
\thanks{This work was supported by the National Natural Science Foundations of China (11771347, 91730306, 41390454, 11271297) and the Science Foundations of Shaanxi Province of
China (2016JQ1029, 2015JM1012).}
\thanks{Manuscript received, ; revised , .}}

%
%

\markboth{Journal of \LaTeX\ Class Files,~Vol.~, No.~, ~}%
{Shell \MakeLowercase{\textit{et al.}}: Bare Demo of IEEEtran.cls for IEEE Journals}
%



\maketitle

\begin{abstract}
In our latest work, a non-convex fraction function is studied to approximate the $\ell_{0}$-norm in $\ell_{0}$-norm minimization problem and translate this NP-hard $\ell_{0}$-norm minimization 
problem into a fraction function minimization problem. Two schemes of iterative FP thresholding algorithms are generated to solve the regularized fraction function minimization problem $(FP^{\lambda}_{a})$. 
One is iterative FP thresholding algorithm-Scheme 1 and the other is iterative FP thresholding algorithm-Scheme 2. A convergence theorem is proved for the iterative FP thresholding algorithm-Scheme 1. 
However, the stationary point generated by the iterative FP thresholding algorithm-Scheme 1 may be a local minimizer due to the non-convexity
of fraction function. Moreover, the regularized parameter $\lambda$ and parameter $a$ in iterative FP thresholding algorithm-Scheme 1 are two fixed given parameters, and how to choose 
the proper regularized parameter $\lambda$ and parameter $a$ for the iterative FP thresholding algorithm-Scheme 1 is a very hard problem. Although the iterative FP thresholding 
algorithm-Scheme 2 is adaptive for the choice of the regularized parameter $\lambda$, the parameter $a$ which influences the behaviour of non-convex fraction function, needs to be 
determined manually in every simulation, and how to determine the proper parameter $a$ is still a very hard problem. In this paper, instead, we will generate a convex iterative FP 
thresholding  algorithm to solve the problem $(FP^{\lambda}_{a})$. Similarly, two schemes of convex iterative FP thresholding algorithms are generated. One is convex iterative FP 
thresholding algorithm-Scheme 1 and the other is convex iterative FP thresholding algorithm-Scheme 2. A global convergence theorem is proved for the convex iterative FP thresholding 
algorithm-Scheme 1. Under an adaptive rule, the convex iterative FP thresholding algorithm-Scheme 2 will be adaptive both for the choice of the regularized parameter $\lambda$ and 
parameter $a$. These are the advantages for our two schemes of convex iterative FP thresholding algorithm compared with our previous proposed two schemes of iterative FP thresholding 
algorithm. At last, we provide a series of numerical simulations to test the performance of the convex iterative FP thresholding algorithm-Scheme 2, and the simulation results show that 
our convex iterative FP thresholding algorithm-Scheme 2 performs very well in recovering a sparse signal.
\end{abstract}


%
\IEEEpeerreviewmaketitle

\section{Introduction}\label{section1}

In information processing, many practical problems can be formulated as the following $\ell_{0}$-minimization problem\cite{Donoho1,Bruckstein2,Ranjan3,Li2019}:
\begin{equation}\label{equ1}
(P_{0})\ \ \ \ \ \min_{\mathbf{x}\in \mathbb{R}^{n}}\|\mathbf{x}\|_{0}\ \ \mathrm{subject}\ \mathrm{to}\ \ \mathbf{A}\mathbf{x}=\mathbf{b},
\end{equation}
where $\mathbf{A}\in \mathbb{R}^{m\times n}$ is a real matrix of full row rank with $m\ll n$, $\mathbf{b}\in \mathbb{R}^{m}$ is a nonzero real column vector, and
$\|\mathbf{x}\|_{0}$ is the $\ell_0$-norm of real vector $\mathbf{x}\in \mathbb{R}^{n}$, which counts the number of the non-zero entries in vector $\mathbf{x}$ \cite{Elad2010,Peng2015}.
The problem $(P_{0})$ aims to seek the sparsest signals which satisfy the underdetermined linear equations. However, it is NP-hard \cite{Natarajan1995,Foucart2013} because
of the discrete and discontinuous nature of the $\ell_{0}$-norm.

In our latest work \cite{Li2019}, we substitute the discontinuous $\ell_{0}$-norm $\|\mathbf{x}\|_{0}$ by the sparsity promoting penalty function
\begin{equation}\label{equ2}
P_{a}(\mathbf{x})=\sum_{i=1}^{n}\rho_{a}(\mathbf{x}_{i})=\sum_{i=1}^{n}\frac{a|\mathbf{x}_{i}|}{a|\mathbf{x}_{i}|+1},\ \ \ a>0,
\end{equation}
where
\begin{equation}\label{equ3}
\rho_{a}(t)=\frac{a|t|}{a|t|+1}
\end{equation}
is the fraction function and concave in $t\in[0,+\infty]$. It is easy to verify that $\rho_{a}(t)=0$ if $t=0$ and $\lim_{a\rightarrow+\infty}\rho_{a}(t)=1$ if $t\neq0$. Clearly, with
the adjustment of parameter $a$, the continuous function $P_{a}(\mathbf{x})$ can approximate the $\ell_{0}$-norm well. Then, we can translate the NP-hard problem $(P_{0})$ into the following
fraction function minimization problem
\begin{equation}\label{equ4}
(FP_{a})\ \ \ \min_{\mathbf{x}\in \mathbb{R}^{n}}P_{a}(\mathbf{x})\ \ \mathrm{subject}\ \mathrm{to}\ \ \mathbf{A}\mathbf{x}=\mathbf{b}
\end{equation}
for the constrained form and
\begin{equation}\label{equ5}
(FP^{\lambda}_{a})\ \ \ \min_{\mathbf{x}\in \mathbb{R}^{n}}\Big\{\|\mathbf{A}\mathbf{x}-\mathbf{b}\|_{2}^{2}+\lambda P_{a}(\mathbf{x})\Big\}
\end{equation}
for the regularized form, where $\lambda>0$ is the regularized parameter.

In \cite{Li2019}, two schemes of iterative FP thresholding algorithm are proposed to solve the problem $(FP^{\lambda}_{a})$. One is iterative FP thresholding algorithm-Scheme 1 and the other
is iterative FP thresholding algorithm-Scheme 2. A large number of numerical simulations on some sparse signal recovery problems have shown that the iterative FP thresholding algorithm-Scheme 2
performances very well in recovering a sparse signal compared with some state-of-art methods. However, the stationary point generated by the iterative FP thresholding algorithm-Scheme 1 may be 
a local minimizer due to the non-convexity of fraction function. Moreover, the regularized parameter $\lambda$ and parameter $a$ in iterative FP thresholding algorithm-Scheme 1 are two fixed 
given parameters, and how to choose the proper regularized parameter $\lambda$ and parameter $a$ for the iterative FP thresholding algorithm-Scheme 1 is a very hard problem. Although the 
iterative FP thresholding algorithm-Scheme 2 is adaptive for the choice of the regularized parameter $\lambda$, the parameter $a$ which influences the behaviour of non-convex fraction function, 
needs to be determined manually in every simulation, and how to determine the proper parameter $a$ is still a very hard problem. In this paper, instead, we will generate a convex iterative FP 
thresholding algorithm to solve the problem $(FP^{\lambda}_{a})$. Similarly, two schemes of convex iterative
FP thresholding algorithms are generated. One is convex iterative FP thresholding algorithm-Scheme 1 and the other is convex iterative FP thresholding algorithm-Scheme 2. A global convergence
theorem is proved for the convex iterative FP thresholding algorithm-Scheme 1. Under an adaptive rule, the convex iterative FP thresholding algorithm-Scheme 2 will be adaptive both for the
choice of the regularized parameter $\lambda$ and parameter $a$. These are the advantages for our two schemes of convex iterative FP thresholding algorithm compared with our previous proposed
two schemes of iterative FP thresholding algorithm.

The rest of this paper is organized as follows. In Section \ref{section2}, we review some known results about our previous proposed iterative FP thresholding algorithm for solving the problem
$(FP^{\lambda}_{a})$. In Section \ref{section3}, a convex iterative FP thresholding algorithm is proposed to solve the regularized problem $(FP^{\lambda}_{a})$. In Section \ref{section4}, a series
of numerical simulations on some sparse signal recovery problems are demonstrated. In \ref{section5}, we conclude some remarks in this paper.

\section{Iterative FP thresholding algorithm for solving the problem $(FP^{\lambda}_{a})$} \label{section2}

In this section, we review some known results from our latest work \cite{Li2019} for our previous proposed two schemes of iterative FP thresholding algorithms to solve the problem $(FP^{\lambda}_{a})$.

\begin{lemma}\label{lem1}
(\cite{Li2019}) Define a function of $\beta\in \mathbb{R}$ as
\begin{equation}\label{equ6}
f_{\lambda}(\beta)=(\beta-\gamma)^{2}+\lambda\rho_{a}(\beta)
\end{equation}
where $\gamma\in \mathbb{R}$ and $\lambda>0$, the optimal solution to $\min_{\beta\in \mathbb{R}}f_{\lambda}(\beta)$ can be expressed as
\begin{equation}\label{equ7}
\begin{array}{llll}
\beta^{\ast}&=&h_{\lambda}(\gamma)\\
&=&\left\{
    \begin{array}{ll}
      g_{\lambda}(\gamma), & \ \ \mathrm{if} \ {|\gamma|> t_{\lambda};} \\
      0, & \ \ \mathrm{if} \ {|\gamma|\leq t_{\lambda},}
    \end{array}
  \right.
  \end{array}
\end{equation}
where $g_{\lambda}(\gamma)$ is defined as
\begin{equation}\label{equ8}
g_{\lambda}(\gamma)=\mathrm{sign}(\gamma)\bigg(\frac{\frac{1+a|\gamma|}{3}(1+2\cos(\frac{\phi_{\lambda}(\gamma)}{3}-\frac{\pi}{3}))-1}{a}\bigg)
\end{equation}
with
\begin{equation}\label{equ9}
\phi_{\lambda}(\gamma)=\arccos\Big(\frac{27\lambda a^{2}}{4(1+a|\gamma|)^{3}}-1\Big),
\end{equation}
and the threshold value $t_{\lambda}$ satisfies
\begin{equation}\label{equ10}
t_{\lambda}=\left\{
    \begin{array}{ll}
      \frac{\lambda a}{2}, & \ \ \mathrm{if} \ {\lambda\leq \frac{1}{a^{2}};} \\
      \sqrt{\lambda}-\frac{1}{2a}, & \ \ \mathrm{if} \ {\lambda>\frac{1}{a^{2}}.}
    \end{array}
  \right.
\end{equation}
\end{lemma}

\begin{definition}\label{def1}
(\cite{Foucart2013}) The nonincreasing rearrangement of the vector $\mathbf{x}\in \mathbb{R}^{n}$ is the vector $|\mathbf{x}|\in \mathbb{R}^{n}$ for which
$$|\mathbf{x}|_{1}\geq |\mathbf{x}|_{2}\geq\cdots\geq |\mathbf{x}|_{n}\geq0$$
and there is a permutation $\pi:[n]\rightarrow[n]$ with $|\mathbf{x}|_{i}=|\mathbf{x}_{\pi(i)}|$ for all $i\in[n]$.
\end{definition}

Now, we consider the following regularized function
\begin{equation}\label{equ11}
\mathcal{C}_{\lambda}(\mathbf{x})=\|\mathbf{A}\mathbf{x}-\mathbf{b}\|_{2}^{2}+\lambda P_{a}(\mathbf{x})
\end{equation}
and its surrogate function
\begin{equation}\label{equ12}
\mathcal{C}_{\lambda, \mu}(\mathbf{x}, \mathbf{z})=\mu[\mathcal{C}_{\lambda}(\mathbf{x})-\|\mathbf{A}\mathbf{x}-\mathbf{A}\mathbf{z}\|_{2}^{2}]+\|\mathbf{x}-\mathbf{z}\|_{2}^{2}
\end{equation}
for any $\lambda>0$, $\mu>0$ and $\mathbf{z}\in \mathbb{R}^{n}$.

When we set $0<\mu\leq\|\mathbf{A}\|_{2}^{-2}$, we can get that
$$\|\mathbf{x}-\mathbf{z}\|_{2}^{2}-\mu\|\mathbf{A}\mathbf{x}-\mathbf{A}\mathbf{z}\|_{2}^{2}\geq0.$$
Therefore, we have
\begin{equation}\label{equ13}
\begin{array}{llll}
\mathcal{C}_{\lambda, \mu}(\mathbf{x}, \mathbf{z})&=&\mu\mathcal{C}_{\lambda}(\mathbf{x})-\mu\|\mathbf{A}\mathbf{x}-\mathbf{A}\mathbf{z}\|_{2}^{2}+\|\mathbf{x}-\mathbf{z}\|_{2}^{2}\\
&\geq&\mu\mathcal{C}_{\lambda}(\mathbf{x}).
\end{array}
\end{equation}

Under the condition $0<\mu\leq\|\mathbf{A}\|_{2}^{-2}$, if we suppose that the vector $\mathbf{x}^{\ast}\in \mathbb{R}^{n}$ is a minimizer of $\mathcal{C}_{\lambda}(\mathbf{x})$, then
\begin{eqnarray*}
\mathcal{C}_{\lambda, \mu}(\mathbf{x},\mathbf{x}^{\ast})&=&\mu[\mathcal{C}_{\lambda}(\mathbf{x})-\|\mathbf{A}\mathbf{x}-\mathbf{A}\mathbf{x}^{\ast}\|_{2}^{2}]+\|\mathbf{x}-\mathbf{x}^{\ast}\|_{2}^{2}\\
&\geq&\mu \mathcal{C}_{\lambda}(\mathbf{x})\\
&\geq&\mu \mathcal{C}_{\lambda}(\mathbf{x}^{\ast})\\
&=&\mathcal{C}_{\lambda, \mu}(\mathbf{x}^{\ast},\mathbf{x}^{\ast}),
\end{eqnarray*}
which implies that $\mathbf{x}^{\ast}$ is also a minimizer of $\mathcal{C}_{\lambda, \mu}(\mathbf{x},\mathbf{x}^{\ast})$ on $\mathbf{x}\in \mathbb{R}^{n}$ for any fixed $\mu\in(0,\|\mathbf{A}\|_{2}^{-2}]$.

On the other hand, $\mathcal{C}_{\lambda, \mu}(\mathbf{x},\mathbf{z})$ can be reexpressed as
\begin{eqnarray*}
\mathcal{C}_{\lambda,\mu}(\mathbf{x},\mathbf{z})&=&\|\mathbf{x}-(\mathbf{z}-\mu \mathbf{A}^{\top}\mathbf{A}\mathbf{z}+\mu \mathbf{A}^{\top}\mathbf{b})\|_{2}^{2}\\
&&+\lambda\mu P_{a}(\mathbf{x})+\mu\|\mathbf{b}\|_{2}^{2}+\|\mathbf{z}\|_{2}^{2}-\mu\|\mathbf{A}\mathbf{z}\|_{2}^{2}\\
&&-\|\mathbf{z}-\mu \mathbf{A}^{\top}\mathbf{A}\mathbf{z}+\mu \mathbf{A}^{\top}\mathbf{b}\|_{2}^{2}\\
&=&\sum_{i=1}^{n}\Big(\big(\mathbf{x}_{i}-(B_{\mu}(\mathbf{z}))_{i}\big)^{2}+\lambda\mu \rho_{a}(\mathbf{x}_{i})\Big)\\
&&+\mu\|\mathbf{b}\|_{2}^{2}+\|\mathbf{z}\|_{2}^{2}-\mu\|\mathbf{A}\mathbf{z}\|_{2}^{2}-\|B_{\mu}(\mathbf{z})\|_{2}^{2},
\end{eqnarray*}
where $B_{\mu}(\mathbf{z})=\mathbf{z}+\mu \mathbf{A}^{\top}(\mathbf{b}-\mathbf{A}\mathbf{z})$. This means that for any fixed $\lambda>0$ and $\mu>0$ minimizing
$\mathcal{C}_{\lambda,\mu}(\mathbf{x},\mathbf{z})$ on $\mathbf{x}\in \mathbb{R}^{n}$ is equivalent to solve
\begin{equation}\label{equ14}
\min_{\mathbf{x}\in \mathbb{R}^{n}}\bigg\{\sum_{i=1}^{n}\Big(\big(\mathbf{x}_{i}-(B_{\mu}(\mathbf{z}))_{i}\big)^{2}+\lambda\mu \rho_{a}(\mathbf{x}_{i})\Big)\bigg\}.
\end{equation}
Notice that the summation of problem (\ref{equ14}) is separable, therefore, solving problem (\ref{equ14}) is equivalent to solving the following $n$ subproblem
\begin{equation}\label{equ15}
\min_{\mathbf{x}_{i}\in \mathbb{R}}\Big\{\big(\mathbf{x}_{i}-(B_{\mu}(\mathbf{z}))_{i}\big)^{2}+\lambda\mu \rho_{a}(\mathbf{x}_{i})\Big\}
\end{equation}
for $i=1,2,\cdots,n$.

By Lemma \ref{lem1}, the minimizer $\mathbf{x}_{i}^{s}\in \mathbb{R}$ of problem (\ref{equ15}) can be given by
\begin{equation}\label{equ16}
\begin{array}{llll}
\mathbf{x}_{i}^{s}&=&h_{\lambda\mu}((B_{\mu}(\mathbf{z}))_{i})\\
&=&\left\{
    \begin{array}{ll}
      g_{\lambda\mu}((B_{\mu}(\mathbf{z}))_{i}), & \mathrm{if} \ {|(B_{\mu}(\mathbf{z}))_{i}|> t_{\lambda\mu};} \\
      0, & \mathrm{if} \ {|(B_{\mu}(\mathbf{z}))_{i}|\leq t_{\lambda\mu},}
    \end{array}
  \right.
  \end{array}
\end{equation}
for $i=1,2,\cdots,n$, where $h_{\lambda\mu}$, $g_{\lambda\mu}$ and $t_{\lambda\mu}$ are obtained by replacing $\lambda$ with $\lambda\mu$ in $h_{\lambda}$, $g_{\lambda}$
and $t_{\lambda}$. Therefore, we get the minimizer $\mathbf{x}^{s}\in \mathbb{R}^{n}$ of $\mathcal{C}_{\lambda,\mu}(\mathbf{x},\mathbf{z})$ on $\mathbf{x}\in \mathbb{R}^{n}$
as follows
\begin{equation}\label{equ17}
\mathbf{x}^{s}=H_{\lambda\mu}(B_{\mu}(\mathbf{z})),
\end{equation}
where $H_{\lambda\mu}$ is the fraction function thresholding operator defined as
\begin{equation}\label{equ18}
H_{\lambda\mu}(B_{\mu}(\mathbf{z}))
=\big(h_{\lambda\mu}((B_{\mu}(\mathbf{z}))_{1}),\cdots,h_{\lambda\mu}((B_{\mu}(\mathbf{z}))_{n})\big)^{\top}.
\end{equation}

The above analysis show us that the regularized problem $(FP^{\lambda}_{a})$ permits a thresholding representation theory for its solution, and through these analysis we can
conclude the following theorem.

\begin{theorem} \label{th1}
For any fixed $\lambda>0$ and $\mu\in(0,\|\mathbf{A}\|_{2}^{-2}]$, if $\mathbf{x}^{\ast}\in \mathbb{R}^{n}$ is a solution to the regularized problem $(FP^{\lambda}_{a})$, then
\begin{equation}\label{equ19}
\mathbf{x}^{\ast}=H_{\lambda\mu}(B_{\mu}(\mathbf{x}^{\ast})).
\end{equation}
Particularly, one can express
\begin{equation}\label{equ20}
\begin{array}{llll}
\mathbf{x}_{i}^{\ast}&=&h_{\lambda\mu}((B_{\mu}(\mathbf{x}^{\ast}))_{i})\\
&=&\left\{
    \begin{array}{ll}
      g_{\lambda\mu}((B_{\mu}(\mathbf{x}^{\ast}))_{i}), & \mathrm{if} \ {|(B_{\mu}(\mathbf{x}^{\ast}))_{i}|> t_{\lambda\mu};} \\
      0, & \mathrm{if} \ {|(B_{\mu}(\mathbf{x}^{\ast}))_{i}|\leq t_{\lambda\mu},}
    \end{array}
  \right.
  \end{array}
\end{equation}
for $i=1,2,\cdots,n$, where
\begin{equation}\label{equ21}
t_{\lambda\mu}=\left\{
    \begin{array}{ll}
      \frac{\lambda\mu a}{2}, & \ \ \mathrm{if} \ {\lambda\leq \frac{1}{a^{2}\mu};} \\
      \sqrt{\lambda\mu}-\frac{1}{2a}, & \ \ \mathrm{if} \ {\lambda>\frac{1}{a^{2}\mu}.}
    \end{array}
  \right.
\end{equation}
\end{theorem}

With the representation (\ref{equ19}), the iterative FP thresholding algorithm for solving the problem $(FP^{\lambda}_{a})$ can be naturally given by
\begin{equation}\label{equ22}
\mathbf{x}^{k+1}=H_{\lambda\mu}(B_{\mu}(\mathbf{x}^{k})). 
\end{equation}

As we all know, the quality of the solutions to the regularized problems depends seriously on the setting of regularized parameter $\lambda$. However, how to choose the optimal regularized 
parameter for a regularized problem is a very hard problem. In \cite{Li2019}, a rule is given to select the proper regularized parameter $\lambda$ for the iterative FP thresholding algorithm. 
In the following, we briefly describe the selection process.

Suppose that the vector $\mathbf{x}^{\ast}$ of sparsity $r$ is the optimal solution to the problem $(FP^{\lambda}_{a})$. Then, by Theorem \ref{th1}, the following inequalities hold:
$$|B_\mu(\mathbf{x}^{\ast})|_{i}>t_{\lambda\mu}\Leftrightarrow i\in \{1,2,\cdot\cdot\cdot,r\},$$
$$|B_\mu(\mathbf{x}^{\ast})|_{j}\leq t\Leftrightarrow j\in \{r+1, r+2, \cdot\cdot\cdot, n\},$$
where $t_{\lambda\mu}$ is the threshold value which is defined in (\ref{equ21}).
According to $\sqrt{\lambda\mu}-\frac{1}{2a}\leq \frac{\lambda\mu a}{2}$, we have

\begin{equation}\label{equ23}
\left\{
  \begin{array}{ll}
   |B_\mu(\mathbf{x}^{\ast})|_{r}>\sqrt{\lambda\mu}-\frac{1}{2a};\\
   |B_\mu(\mathbf{x}^{\ast})|_{r+1}\leq\frac{\lambda\mu a}{2},
  \end{array}
\right.
\end{equation}
which implies
\begin{equation}\label{equ24}
\frac{2|B_\mu(\mathbf{x}^{\ast})|_{r+1}}{a\mu}\leq\lambda <\frac{(2a|B_\mu(\mathbf{x}^{\ast})|_{r}+1)^2}{4a^{2}\mu}.
\end{equation}
The above estimate helps to set the proper regularized parameter $\lambda$. A choice of the proper regularized parameter $\lambda$ is
\begin{equation}\label{equ25}
\begin{array}{llll}
\lambda^{\ast}=\left\{
            \begin{array}{ll}
              \lambda_{1}=\frac{2|B_{\mu}(\mathbf{x}^{\ast})|_{r+1}}{a\mu}, & {\mathrm{if} \ \lambda_{1}\leq\frac{1}{a^{2}\mu};} \\
              \lambda_{2}=\frac{(1-\epsilon)(2a|B_{\mu}(\mathbf{x}^{\ast})|_{r}+1)^{2}}{4a^{2}\mu}, & {\mathrm{if} \ \lambda_{1}>\frac{1}{a^{2}\mu},}
            \end{array}
          \right.
\end{array}
\end{equation}
where $\epsilon$ is a small positive number such as 0.1,0.01 or 0.001.

Instead of real solution, we can approximate $\mathbf{x}^{\ast}$ by known $\mathbf{x}^{k}$ in (\ref{equ25}). Therefore, in each iteration, the proper regularized
parameter $\lambda$ can be selected as
\begin{equation}\label{equ26}
\begin{array}{llll}
\lambda_{k}^{\ast}=\left\{
            \begin{array}{ll}
              \lambda_{1,k}=\frac{2|B_{\mu}(\mathbf{x}^{k})|_{r+1}}{a\mu}, & {\mathrm{if} \ \lambda_{1,k}\leq\frac{1}{a^{2}\mu};} \\
              \lambda_{2,k}=\frac{(1-\epsilon)(2a|B_{\mu}(\mathbf{x}^{k})|_{r}+1)^{2}}{4a^{2}\mu}, & {\mathrm{if} \ \lambda_{1,k}>\frac{1}{a^{2}\mu}.}
            \end{array}
          \right.
\end{array}
\end{equation}
When doing so, the iterative FP thresholding algorithm will be adaptive and free from the choice of the regularized parameter $\lambda$.

Incorporated with different parameter-setting strategies, iteration (\ref{equ22}) defines different implementation schemes of the iterative FP thresholding algorithm, i.e.,

Scheme 1: $\mu=\mu_{0}\in(0, \|\mathbf{A}\|_{2}^{-2})$; $\lambda=\lambda_0\in(\|b\|_{2}^{2}, \bar{\lambda})$, \footnote{
The parameter $\bar{\lambda}$ in Scheme 1 is defined as
$$\bar{\lambda}=\|b\|_{2}^{2}+\frac{\|A^{\top}b\|_{\infty}+\sqrt{\|A^{\top}b\|_{\infty}+2a\|b\|_{2}^{2}\|A^{\top}b\|_{\infty}}}{a}.$$} and $a=a_{0}$ ($a_{0}>0$
is a given positive number).

Scheme 2: $\mu=\mu_{0}\in(0, \|\mathbf{A}\|_{2}^{-2})$; $\lambda=\lambda_{k}^{\ast}$ defined in (\ref{equ26}), and $a=a_{0}$ ($a_{0}>0$ is a given positive number).

There is one more thing in Scheme 2 needed to be mentioned that, in each iteration, the threshold value function is set to $t_{\lambda\mu}=\frac{\lambda\mu a}{2}$ when
$\lambda=\lambda_{1,k}$, and $t_{\lambda\mu}=\sqrt{\lambda\mu}-\frac{1}{2a}$ when $\lambda=\lambda_{2,k}$.

\begin{algorithm}
\caption{: Iterative FP Thresholding Algorithm-Scheme 1 (IFPTA-S1)}
\label{alg:A}
\begin{algorithmic}
\STATE {\textbf{Input}: $\mathbf{A}\in\mathbb{R}^{m\times n}$, $\mathbf{b}\in \mathbb{R}^{m}$, $\mu=\mu_{0}\in(0,\|\mathbf{A}\|_{2}^{-2})$, $\lambda=\lambda_0\in(\|b\|_{2}^{2}, \bar{\lambda})$,
$a=a_{0}$ ($a_{0}>0$ is a given positive number);}
\STATE {\textbf{Initialize}: Given $\mathbf{x}^{0}\in \mathbb{R}^{n}$;}
\STATE {\textbf{while} not converged \textbf{do}}
\STATE \ {$B_{\mu}(\mathbf{x}^{k})=\mathbf{x}^{k}+\mu \mathbf{A}^{\top}(\mathbf{b}-\mathbf{A}\mathbf{x}^{k})$;}
\STATE \ {\textbf{if}\ $\lambda\leq\frac{1}{a^{2}\mu}$\ then}
\STATE \ \ \ \  {$t_{\lambda\mu}=\frac{\lambda\mu a}{2}$}
\STATE \ \ \ \ {for\ $i=1:\mathrm{length}(\mathbf{x})$}
\STATE \ \ \ \ \ \  {if $|(B_{\mu}(\mathbf{x}^{k}))_{i}|>t_{\lambda\mu}$, then $\mathbf{x}_{i}^{k+1}=g_{\lambda\mu}((B_{\mu}(\mathbf{x}^{k}))_{i})$;}
\STATE \ \ \ \ \ \  {if $|(B_{\mu}(\mathbf{x}^{k}))_{i}|\leq t_{\lambda\mu}$, then $\mathbf{x}_{i}^{k+1}=0$;}
\STATE \  {\textbf{else}}
\STATE \ \ \ \ {$t_{\lambda\mu}=\sqrt{\lambda\mu}-\frac{1}{2a}$}
\STATE \ \ \ \ {for\ $i=1:\mathrm{length}(\mathbf{x})$}
\STATE \ \ \ \ \ \ {if $|(B_{\mu}(\mathbf{x}^{k}))_{i}|>t_{\lambda\mu}$, then $\mathbf{x}_{i}^{k+1}=g_{\lambda\mu}((B_{\mu}(\mathbf{x}^{k}))_{i})$;}
\STATE \ \ \ \ \ \ {if $|(B_{\mu}(\mathbf{x}^{k}))_{i}|\leq t_{\lambda\mu}$, then $\mathbf{x}_{i}^{k+1}=0$;}
\STATE \ \  {\textbf{end}}
\STATE \ \  {$k\rightarrow k+1$;}
\STATE{\textbf{end while}}
\STATE{\textbf{Output}: $\mathbf{x}^{\ast}$}
\end{algorithmic}
\end{algorithm}

The above analysis leads to two schemes of the iterative FP thresholding algorithm. One is iterative FP Thresholding Algorithm-Scheme 1 (which is named IFPTA-S1
for short in this paper) and the other is iterative FP Thresholding Algorithm-Scheme 2 (which is named IFPTA-S2 for short in this paper). We summarized these two
algorithms in Algorithm \ref{alg:A} and Algorithm \ref{alg:B}.

The basic convergence theorem of IFPTA-S1 can be stated as below

\begin{theorem} \label{th2}
Let $\{\mathbf{x}^{k}\}$ be the sequence generated by IFPTA-S1. Then \\
(1) The sequence $\{\mathcal{C}_{\lambda}(\mathbf{x}^{k})\}$ is decreasing;\\
(2) $\{\mathbf{x}^{k}\}$ is asymptotically regular, i.e., $\lim_{k\rightarrow\infty}\|\mathbf{x}^{k+1}-\mathbf{x}^{k}\|_{2}^{2}=0$;\\
(3) $\{\mathbf{x}^{k}\}$ converges to a stationary point of the iteration (\ref{equ22}).
\end{theorem}

Due to the non-convexity of the objective function in (\ref{equ14}), the sequence $\{\mathbf{x}^{k}\}$ generated by IFPTA-S1 always converges
to a local minimizer of the non-convex problem (\ref{equ14}). How to find the unique global minimizer for the non-convex minimization problem (\ref{equ14}) is a really challenging
problem. Moreover, we can find that the regularized parameter $\lambda$ and parameter $a$ in IFPTA-S1 are two fixed given parameters, and how to choose the best regularized parameter $\lambda$
and parameter $a$ for the iterative FP thresholding algorithm-Scheme 1 is a very hard problem. Although the iterative FP thresholding algorithm-Scheme 2 is adaptive for the choice of the
regularized parameter $\lambda$, the parameter $a$ which influences the behaviour of non-convex fraction function, needs to be determined manually in every simulation, and how to determine
the proper parameter $a$ is still a very hard problem.

\begin{algorithm}[h!]
\caption{: Iterative FP thresholding algorithm-Scheme 2 (IFPTA-S2)}
\label{alg:B}
\begin{algorithmic}
\STATE {\textbf{Input}: $\mathbf{A}\in\mathbb{R}^{m\times n}$, $\mathbf{b}\in \mathbb{R}^{m}$, $\mu=\mu_{0}\in(0,\|\mathbf{A}\|_{2}^{-2})$, $a=a_{0}$ ($a_{0}>0$ is a given positive number) and $\epsilon>0$
($\epsilon$ is a small positive number such as 0.1,0.01 or 0.001);}
\STATE {\textbf{Initialize}: Given $\mathbf{x}^{0}\in \mathbb{R}^{n}$;}
\STATE {\textbf{while} not converged \textbf{do}}
\STATE \ {$B_{\mu}(\mathbf{x}^{k})=\mathbf{x}^{k}+\mu \mathbf{A}^{\top}(\mathbf{b}-\mathbf{A}\mathbf{x}^{k})$;}
\STATE \ {$\lambda_{1,k}=\frac{2|B_{\mu}(\mathbf{x}^{k})|_{r+1}}{a\mu}$, $\lambda_{2,k}=\frac{(1-\epsilon)(2a|B_{\mu}(\mathbf{x}^{k})|_{r}+1)^{2}}{4a^{2}\mu}$;}
\STATE \ {\textbf{if}\ $\lambda_{1,k}\leq\frac{1}{a^{2}\mu}$\ \textbf{then}}
\STATE \ \ \ \ {$\lambda=\lambda_{1,k}$, $t_{\lambda\mu}=\frac{\lambda\mu a}{2}$;}
\STATE \ \ \ \  {for\ $i=1:\mathrm{length}(\mathbf{x})$}
\STATE \ \ \ \ \ \ {if $|(B_{\mu}(\mathbf{x}^{k}))_{i}|>t_{\lambda\mu}$, then $\mathbf{x}_{i}^{k+1}=g_{\lambda\mu}((B_{\mu}(\mathbf{x}^{k}))_{i})$;}
\STATE \ \ \ \ \ \ {if $|(B_{\mu}(\mathbf{x}^{k}))_{i}|\leq t_{\lambda\mu}$, then $\mathbf{x}_{i}^{k+1}=0$;}
\STATE \ {\textbf{else}}
\STATE \ \ \ \ {$\lambda=\lambda_{2,k}$, $t_{\lambda\mu}=\sqrt{\lambda\mu}-\frac{1}{2a}$;}
\STATE \ \ \ \ {for\ $i=1:\mathrm{length}(\mathbf{x})$}
\STATE \ \ \ \ \ \ {if $|(B_{\mu}(\mathbf{x}^{k}))_{i}|>t_{\lambda\mu}$, then $\mathbf{x}_{i}^{k+1}=g_{\lambda\mu}((B_{\mu}(\mathbf{x}^{k}))_{i})$;}
\STATE \ \ \ \ \ \ {if $|(B_{\mu}(\mathbf{x}^{k}))_{i}|\leq t_{\lambda\mu}$, then $\mathbf{x}_{i}^{k+1}=0$;}
\STATE \ \ {\textbf{end}}
\STATE \ \ {$k\rightarrow k+1$;}
\STATE{\textbf{end while}}
\STATE{\textbf{Output}: $\mathbf{x}^{\ast}$}
\end{algorithmic}
\end{algorithm}

\section{Convex iterative FP thresholding algorithm for solving the problem $(FP^{\lambda}_{a})$} \label{section3}

In this section, we will generate a convex iterative FP thresholding algorithm to solve the problem $(FP^{\lambda}_{a})$. Before we giving out the analytic expression of our convex iterative
FP thresholding algorithm, some crucial results need to be proved for later use.

\begin{theorem}\label{the3}
For any $0<a\leq\frac{1}{\sqrt{\lambda}}$, the function $f_{\lambda}(\beta)$ defined in equation (\ref{equ6}) is strictly convex.
\end{theorem}

\begin{proof} It is clear to see that the function
$$f_{\lambda}(\beta)=(\beta-\gamma)^{2}+\lambda\rho_{a}(\beta)$$
is differentiable on $\mathbb{R}\backslash\{0\}$. For $\beta\neq 0$, the derivative of $f_{\lambda}(\beta)$ is given by
\begin{equation}\label{equ27}
f^{\prime}_{\lambda}(\beta)=2(\beta-\gamma)+\frac{\lambda a}{(a|\beta|+1)^{2}}\mathrm{sign}(\beta), \ \ \beta\neq 0.
\end{equation}
Let us find the range of $a$ for which $f_{\lambda}(\beta)$ is convex. Notice that
\begin{equation}\label{equ28}
\rho^{\prime}_{a}(0^{+})=a>-a=\rho^{\prime}_{a}(0^{-})
\end{equation}
and
\begin{equation}\label{equ29}
f^{\prime\prime}_{\lambda}(\beta)=2-\frac{2\lambda a^{2}}{(a\beta+1)^{3}}, \ \ \ \forall \ \beta>0.
\end{equation}
When we set $0<a\leq\frac{1}{\sqrt{\lambda}}$, the function $f^{\prime}_{\lambda}(\beta)$ defined in (\ref{equ27}) is increasing, then the function $f_{\lambda}(\beta)$ in (\ref{equ6})
is strictly convex. This completes the proof.
\end{proof}

\begin{figure}[h!]
 \centering
 \includegraphics[width=0.4\textwidth]{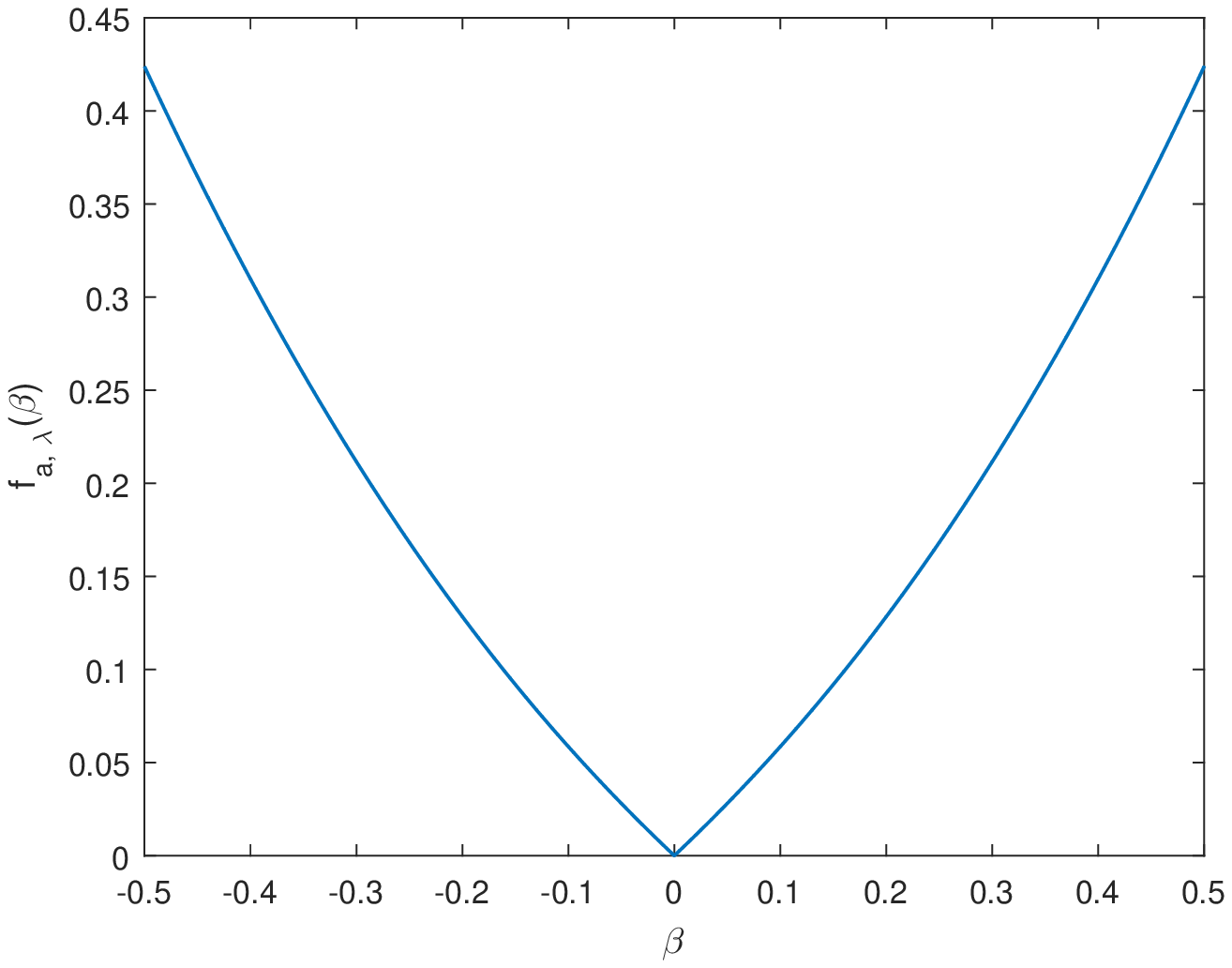}
\caption{The behavior of the function $f_{\lambda}(\beta)=(\beta-\gamma)^{2}+\lambda\rho_{a}(\beta)$ with $\lambda=0.49$, $a=1.1$ and $\gamma=0$.}
\label{fig:1}
\end{figure}

The graph presented in Fig. \ref{fig:1} show the plot of the function $f_{\lambda}(\beta)=(\beta-\gamma)^{2}+\lambda\rho_{a}(\beta)$ with $\lambda=0.49$, $a=1.1$
and $\gamma=0$. It can be seen in Fig.\ref{fig:1} that the function $f_{\lambda}(\beta)=(\beta-\gamma)^{2}+\lambda\rho_{a}(\beta)$ is strictly convex for $\lambda=0.49$
and $a=1.1$, even though the penalty function $\rho_{a}(\beta)$ is not convex. However, when $\lambda=0.49$ and $a=50$, the graph presented in Fig. \ref{fig:2}
show that the function $f_{\lambda}(\beta)=(\beta-\gamma)^{2}+\lambda\rho_{a}(\beta)$ is non-convex.

Theorem \ref{the3} tell us that the convexity of the function $f_{\lambda}(\beta)$ defined in (\ref{equ6}) can be ensured by constraining the parameter $a$ as $0<a\leq\frac{1}{\sqrt{\lambda}}$
in the non-convex fraction function $\rho_{a}(\beta)$, which means that there exist the unique global minimizer to $f_{\lambda}(\beta)$ when we set $0<a\leq\frac{1}{\sqrt{\lambda}}$. Combined with 
Lemma \ref{lem1}, the unique global minimizer to $f_{\lambda}(\beta)$ can be expressed in the following theorem.

\begin{figure}[h!]
 \centering
 \includegraphics[width=0.4\textwidth]{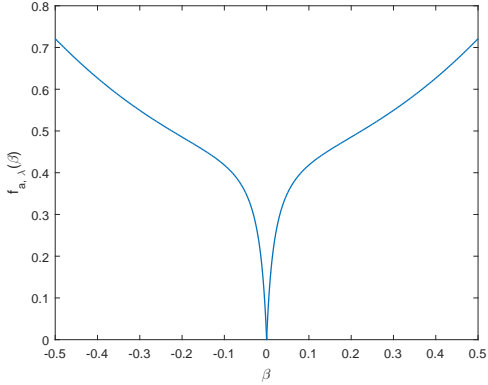}
\caption{The behavior of the function $f_{\lambda}(\beta)=(\beta-\gamma)^{2}+\lambda\rho_{a}(\beta)$ with $\lambda=0.49$, $a=50$ and $\gamma=0$.}
\label{fig:2}
\end{figure}

\begin{theorem}\label{the4}
Suppose $0<a\leq\frac{1}{\sqrt{\lambda}}$, the unique global optimal solution to $\min_{\beta\in \mathbb{R}}f_{\lambda}(\beta)$ can be given by
\begin{equation}\label{equ30}
\begin{array}{llll}
\beta^{\ast}&=&h_{\lambda}(\gamma)\\
&=&\left\{
    \begin{array}{ll}
      g_{\lambda}(\gamma), & \ \ \mathrm{if} \ {|\gamma|>t_{\lambda};} \\
      0, & \ \ \mathrm{if} \ {|\gamma|\leq t_{\lambda}.}
    \end{array}
  \right.
  \end{array}
\end{equation}
where $g_{\lambda}$ is defined in (\ref{equ8}) and the threshold value $t_{\lambda}$ satisfies
\begin{equation}\label{equ31}
t_{\lambda}=\frac{\lambda a}{2}.
\end{equation}
\end{theorem}

\begin{proof}
We can see that the condition $0<a\leq\frac{1}{\sqrt{\lambda}}$ implies that $\lambda\leq\frac{1}{a^{2}}$. Under this condition, the threshold value $t_{\lambda}$ defined in (\ref{equ10}) is
simplified to just one threshold parameter $\frac{\lambda a}{2}$. The `unique global' is obtained by the strictly convexity of the function $f_{\lambda}(\beta)$ under the condition
$0<a\leq\frac{1}{\sqrt{\lambda}}$. This completes the proof.
\end{proof}

In addition, by Theorem \ref{the3}, we can also get that the objective functions in (\ref{equ15}) and (\ref{equ14}) are all strictly convex functions under the condition
$0<a\leq\frac{1}{\sqrt{\lambda\mu}}$. Therefore, for any fixed $\lambda>0$, $\mu>0$ and $a\in(0,\frac{1}{\sqrt{\lambda\mu}}]$, there exist the unique global minimizers to the
problems (\ref{equ15}) and (\ref{equ14}).

In the following, we shall present a convex iterative FP thresholding algorithm to solve the problem $(FP^{\lambda}_{a})$ under the condition $0<a\leq\frac{1}{\sqrt{\lambda\mu}}$.

Similar as the generation of the iterative FP thresholding algorithm (iteration (\ref{equ22})), under the condition $0<a\leq\frac{1}{\sqrt{\lambda\mu}}$, the iterative FP thresholding algorithm for
solving the problem $(FP^{\lambda}_{a})$ can be rewritten as
\begin{equation}\label{equ32}
\mathbf{x}^{k+1}=H_{\lambda\mu}(B_{\mu}(\mathbf{x}^{k})),
\end{equation}
where
\begin{equation}\label{equ33}
H_{\lambda\mu}(B_{\mu}(\mathbf{x}^{k}))
=\big(h_{\lambda\mu}((B_{\mu}(\mathbf{x}^{k}))_{1}),\cdots,h_{\lambda\mu}((B_{\mu}(\mathbf{x}^{k}))_{n})\big)^{\top}
\end{equation}
with
\begin{equation}\label{equ34}
h_{\lambda\mu}((B_{\mu}(\mathbf{x}^{k}))_{i})=\left\{
    \begin{array}{ll}
      g_{\lambda\mu}((B_{\mu}(\mathbf{x}^{k}))_{i}), & \mathrm{if} \ {|(B_{\mu}(\mathbf{x}^{k}))_{i}|>t_{\lambda\mu};} \\
      0, & \mathrm{if} \ {|(B_{\mu}(\mathbf{x}^{k}))_{i}|\leq t_{\lambda\mu}.}
    \end{array}
  \right.
\end{equation}
and the threshold value $t_{\lambda\mu}$ satisfies
\begin{equation}\label{equ35}
t_{\lambda\mu}=\frac{\lambda\mu a}{2}
\end{equation}
which is obtained by replacing $\lambda$ with $\lambda\mu$ in $t_{\lambda}$, and the $t_{\lambda}$ here is defined in (\ref{equ31}). Due to the convexity of the objective 
function in (\ref{equ14}) under the condition $0<a\leq\frac{1}{\sqrt{\lambda\mu}}$, we call the thresholding algorithm which is defined in (\ref{equ32}) the convex iterative 
FP thresholding algorithm.

\begin{remark}\label{remak1}
The $g_{\lambda\mu}$ in this paper are all obtained by replacing $\lambda$ with $\lambda\mu$ in $g_{\lambda}$, i.e.,
\begin{equation}\label{equ36}
g_{\lambda\mu}(\gamma)=
\mathrm{sign}(\gamma)\bigg(\frac{\frac{1+a|\gamma|}{3}(1+2\cos(\frac{\phi_{\lambda\mu}(\gamma)}{3}-\frac{\pi}{3}))-1}{a}\bigg)
\end{equation}
with
\begin{equation}\label{equ37}
\phi_{\lambda\mu}(\gamma)=\arccos\Big(\frac{27\lambda\mu a^{2}}{4(1+a|\gamma|)^{3}}-1\Big).
\end{equation}
\end{remark}

Similar as the iterative FP thresholding algorithm, the the quality of our convex iterative FP thresholding algorithm also depends seriously on the setting of regularized
parameter $\lambda$ and parameter $a$, and how to select the proper parameters $\lambda$ and $a$ in our convex iterative FP thresholding algorithm is also a very hard problem.
In detailed applications, the parameters $\lambda$ and $a$ must be carefully chosen. In the following description, we will generate an adaptive rule for the choice of the
parameters $\lambda$ and $a$ in our convex iterative FP thresholding algorithm. When doing so, our convex iterative FP thresholding algorithm will be adaptive both for the choice of
the regularized parameter $\lambda$ and parameter $a$.

\emph{1) Adaptive for the choice of parameter $a$:} Note that the parameter $a$ in convex iterative FP thresholding algorithm should be satisfied $a\in(0,\frac{1}{\sqrt{\lambda\mu}}]$. Therefore,
we can choose the parameter $a$ as
\begin{equation}\label{equ38}
a=\frac{\tau}{\sqrt{\lambda\mu}},
\end{equation}
where $\tau\in(0,1]$ is a given positive number. When we set $a=\frac{\tau}{\sqrt{\lambda\mu}}$, the threshold value $t_{\lambda\mu}$ in (\ref{equ35}) can be rewritten as
\begin{equation}\label{equ39}
t_{\lambda\mu}=\frac{\tau\sqrt{\lambda\mu}}{2}.
\end{equation}
To see clear that once the value of the regularized parameter $\lambda$ is determined, the parameter $a$ can be given by (\ref{equ38}), and therefore the convex iterative FP thresholding
algorithm will be adaptive for the choice of the parameter $a$. For the choice of the proper regularized parameter $\lambda$ , here, the rule which is used to select the proper regularized 
parameter $\lambda$ in our previous proposed IFPTA-S2 is again used to select the proper regularized parameter $\lambda$ in our convex iterative FP thresholding algorithm. 

\emph{2) Adaptive for the choice of regularized parameter $\lambda$:} Let the vector $\mathbf{x}^{\ast}$ of sparsity $r$ be the optimal solution to the problem $(FP^{\lambda}_{a})$. Then, 
the following inequalities hold
$$|B_{\mu}(x^{\ast})|_{i}>t_{\lambda\mu}=\frac{\tau\sqrt{\lambda\mu}}{2}\Leftrightarrow i\in\{1,2,\cdots,r\},$$
$$|B_{\mu}(x^{\ast})|_{j}\leq t_{\lambda\mu}=\frac{\tau\sqrt{\lambda\mu}}{2}\Leftrightarrow j\in\{r+1,r+2,\cdots,n\}$$
which implies that
\begin{equation}\label{equ40}
\frac{4|B_{\mu}(\mathbf{x}^{\ast})|^{2}_{r+1}}{\tau^{2}\mu}\leq\lambda<\frac{4|B_{\mu}(\mathbf{x}^{\ast})|^{2}_{r}}{\tau^{2}\mu}.
\end{equation}
The above estimation provides an exact location of the regularized parameter $\lambda$. A most reliable choice of the proper regularized parameter $\lambda$ specified by
\begin{equation}\label{equ41}
\begin{array}{llll}
\lambda^{\ast}&=&\displaystyle\frac{4|B_{\mu}(\mathbf{x}^{\ast})|^{2}_{r+1}}{\tau^{2}\mu}\\
&&+\displaystyle\min\bigg\{\zeta, c\Big(\frac{4|B_{\mu}(\mathbf{x}^{\ast})|^{2}_{r}}{\tau^{2}\mu}-\frac{4|B_{\mu}(\mathbf{x}^{\ast})|^{2}_{r+1}}{\tau^{2}\mu}\Big)\bigg\},
\end{array}
\end{equation}
where $\zeta>0$ and $c\in[0, 1)$. Clearly, the parameter $\lambda^{\ast}$ defined in (\ref{equ41}) satisfies the inequalities (\ref{equ40}). Combing with (\ref{equ41}) and (\ref{equ38}), 
the proper parameter $a$ for our convex iterative FP thresholding algorithm can be selected as
\begin{equation}\label{equ42}
a^{\ast}=\frac{\tau}{\sqrt{\lambda^{\ast}\mu}}. 
\end{equation}

In each iteration, we can approximate the optimal solution $\mathbf{x}^{\ast}$ by $\mathbf{x}^{k}$. Then, in each iteration, the proper regularized parameter $\lambda$ and parameter $a$ 
for our convex iterative FP thresholding algorithm can be selected as
\begin{equation}\label{equ43}
\begin{array}{llll}
\lambda_{k}^{\ast}&=&\displaystyle\frac{4|B_{\mu}(\mathbf{x}^{k})|^{2}_{r+1}}{\tau^{2}\mu}\\
&&+\displaystyle\min\bigg\{\zeta, c\Big(\frac{4|B_{\mu}(\mathbf{x}^{k})|^{2}_{r}}{\tau^{2}\mu}-\frac{4|B_{\mu}(\mathbf{x}^{k})|^{2}_{r+1}}{\tau^{2}\mu}\Big)\bigg\}
\end{array}
\end{equation}
and
\begin{equation}\label{equ44}
a^{\ast}_{k}=\frac{\tau}{\sqrt{\lambda^{\ast}_{k}\mu}}. 
\end{equation}

By above operations, our convex iterative FP thresholding algorithm will be adaptive for the choice of the regularized parameter $\lambda$ and parameter $a$ in each iteration.

Similar as the iterative FP thresholding algorithm, incorporated with different parameter-setting strategies, iteration (\ref{equ32}) can also defines different
implementation schemes of the convex iterative FP thresholding algorithm in each iteration, i.e.,

Scheme 1: $\mu=\mu_{0}\in(0, \|\mathbf{A}\|_{2}^{-2})$; $\lambda=\lambda_0$ ($\lambda_0>0$ is a given positive number), and $a=a_{0}\in(0,\frac{1}{\sqrt{\lambda\mu}}]$.

Scheme 2: $\mu=\mu_{0}\in(0, \|\mathbf{A}\|_{2}^{-2})$; $\lambda=\lambda^{\ast}_{k}$ defined in (\ref{equ43}), and $a=a^{\ast}_{k}$ defined in (\ref{equ44}).

It is worth noting that, in some $k$-th iterations, the values of $\lambda^{\ast}_{k}$ defined in (\ref{equ43}) may be equal to zero which lead to the infinite numbers of $a^{\ast}_{k}$.
In fact, in $k$-th iteration, if $\lambda=\lambda^{\ast}_{k}=0$, the fraction function $\rho_{a}$ in (\ref{equ14}) will be vanished for any choice of $a>0$. Under this circumstance,
we can choose the parameter $a$ as a random given positive number $\hat{a}>0$. That is to say, in $k$-th iteration, if $\lambda^{\ast}_{k}$ defined in (\ref{equ43}) equals to zero, we can set
the parameter $a$ as a random given positive number $\hat{a}>0$. 

The above analysis also leads to two schemes of the convex iterative FP thresholding algorithm. One is convex iterative FP thresholding algorithm-Scheme 1 (which is named CIFPTA-S1 for 
short in this paper) and the other is convex iterative FP thresholding algorithm-Scheme 2 (which is named CIFPTA-S2 for short in this paper). We summarized these two algorithms in 
Algorithm \ref{alg:C} and Algorithm \ref{alg:D}.

\begin{algorithm}[h!]
\caption{: Convex iterative FP thresholding algorithm-Scheme 1 (CIFPTA-S1)}
\label{alg:C}
\begin{algorithmic}
\STATE {\textbf{Input}: $\mathbf{A}\in\mathbb{R}^{m\times n}$, $\mathbf{b}\in \mathbb{R}^{m}$, $\mu=\mu_{0}\in(0, \|\mathbf{A}\|_{2}^{-2})$, $\lambda=\lambda_0$ ($\lambda_0>0$ is a 
given positive number) and $a=a_{0}\in(0,\frac{1}{\sqrt{\lambda\mu}}]$;}
\STATE {\textbf{Initialize}: Given $\mathbf{x}^{0}\in \mathbb{R}^{n}$;}
\STATE {\textbf{while} not converged \textbf{do}}
\STATE \ {$B_{\mu}(\mathbf{x}^{k})=\mathbf{x}^{k}+\mu \mathbf{A}^{\top}(\mathbf{b}-\mathbf{A}\mathbf{x}^{k})$;}
\STATE \ {$t_{\lambda\mu}=\frac{\lambda\mu a}{2}$;}
\STATE \ \ \ \ {for\ $i=1:\mathrm{length}(\mathbf{x})$}
\STATE \ \ \ \ \ \ {if $|(B_{\mu}(\mathbf{x}^{k}))_{i}|>t_{\lambda\mu}$, then $\mathbf{x}_{i}^{k+1}=g_{\lambda\mu}((B_{\mu}(\mathbf{x}^{k}))_{i})$;}
\STATE \ \ \ \ \ \ {if $|(B_{\mu}(\mathbf{x}^{k}))_{i}|\leq t_{\lambda\mu}$, then $\mathbf{x}_{i}^{k+1}=0$;}
\STATE \ {$k\rightarrow k+1$;}
\STATE{\textbf{end while}}
\STATE{\textbf{Output}: $\mathbf{x}^{\ast}$}
\end{algorithmic}
\end{algorithm}

\begin{algorithm}[h!]
\caption{: Convex iterative FP thresholding algorithm-Scheme 2 (CIFPTA-S2)}
\label{alg:D}
\begin{algorithmic}
\STATE {\textbf{Input}: $\mathbf{A}\in\mathbb{R}^{m\times n}$, $\mathbf{b}\in \mathbb{R}^{m}$, $\mu=\mu_{0}\in(0, \|\mathbf{A}\|_{2}^{-2})$, $\tau\in(0,1]$, $\hat{a}>0$ is a random given positive number,
$\zeta>0$ and $c\in[0, 1)$;}
\STATE {\textbf{Initialize}: Given $\mathbf{x}^{0}\in \mathbb{R}^{n}$;}
\STATE {\textbf{while} not converged \textbf{do}}
\STATE \ {$B_{\mu}(\mathbf{x}^{k})=\mathbf{x}^{k}+\mu \mathbf{A}^{\top}(\mathbf{b}-\mathbf{A}\mathbf{x}^{k})$;}
\STATE \ {$\lambda^{\ast}_{k}=\frac{4|B_{\mu}(\mathbf{x}^{k})|^{2}_{r+1}}{\tau^{2}\mu}+\min\Big\{\zeta, c\big(\frac{4|B_{\mu}(\mathbf{x}^{k})|^{2}_{r}}{\tau^{2}\mu}-\frac{4|B_{\mu}(\mathbf{x}^{k})|^{2}_{r+1}}{\tau^{2}\mu}\big)\Big\}$;}
\STATE \ \ \ {\textbf{if}\ $\lambda^{\ast}_{k}\neq 0$\ \textbf{then}}
\STATE \ \ \ \ \ \ {$\lambda=\lambda^{\ast}_{k}$, $a=\frac{\tau}{\sqrt{\lambda^{\ast}_{k}\mu}}$, $t_{\lambda\mu}=\frac{\tau\sqrt{\lambda\mu}}{2}$;}
\STATE \ \ \ \ \ \ {for\ $i=1:\mathrm{length}(\mathbf{x})$}
\STATE \ \ \ \ \ \ \ \ {if $|(B_{\mu}(\mathbf{x}^{k}))_{i}|>t_{\lambda\mu}$, then $\mathbf{x}_{i}^{k+1}=g_{\lambda\mu}((B_{\mu}(\mathbf{x}^{k}))_{i})$;}
\STATE \ \ \ \ \ \ \ \ {if $|(B_{\mu}(\mathbf{x}^{k}))_{i}|\leq t_{\lambda\mu}$, then $\mathbf{x}_{i}^{k+1}=0$;}
\STATE \ \ \ \ {\textbf{else}}
\STATE \ \ \ \ \ \ {$\lambda=0$, $a=\hat{a}$, $t_{\lambda\mu}=\frac{\tau\sqrt{\lambda\mu}}{2}$;}
\STATE \ \ \ \ \ \ {for\ $i=1:\mathrm{length}(\mathbf{x})$}
\STATE \ \ \ \ \ \ \ \ {if $|(B_{\mu}(\mathbf{x}^{k}))_{i}|>t_{\lambda\mu}$, then $\mathbf{x}_{i}^{k+1}=g_{\lambda\mu}((B_{\mu}(\mathbf{x}^{k}))_{i})$;}
\STATE \ \ \ \ \ \ \ \ {if $|(B_{\mu}(\mathbf{x}^{k}))_{i}|\leq t_{\lambda\mu}$, then $\mathbf{x}_{i}^{k+1}=0$;}
\STATE \ \ \ \ {\textbf{end}}
\STATE \ {$k\rightarrow k+1$;}
\STATE{\textbf{end while}}
\STATE{\textbf{Output}: $\mathbf{x}^{\ast}$}
\end{algorithmic}
\end{algorithm}

At the end of this section, we justify the convergence of the CIFPTA-S1.

\begin{theorem} \label{th5}
Let $\{\mathbf{x}^{k}\}$ be the sequence generated by CIFPTA-S1. Then\\
(1) The sequence $\{\mathcal{C}_{\lambda}(\mathbf{x}^{k})\}$ is decreasing.\\
(2) $\{\mathbf{x}^{k}\}$ is asymptotically regular, i.e., $\lim_{k\rightarrow\infty}\|\mathbf{x}^{k+1}-\mathbf{x}^{k}\|_{2}^{2}=0$.\\
(3) $\{\mathbf{x}^{k}\}$ converges to the unique global stationary point of iteration (\ref{equ32}).
\end{theorem}

\begin{proof}
The proof of Theorem \ref{th5} is similar to the proof of Theorem 4.1 in \cite{Peng2017}. Due to the convexity of the objective function in (\ref{equ14}), the sequence $\{\mathbf{x}^{k}\}$ converges
to the unique global stationary point of the iteration (\ref{equ32}).
\end{proof}

\section{Numerical Simulations}\label{section4}

In this section, we present a series of numerical simulations on some sparse signal recovery problems to demonstrate the performances of our convex iterative FP thresholding algorithm.
All the numerical simulations here are conducted by applying the CIFPTA-S2 (convex iterative FP thresholding algorithm-Scheme 2). In these numerical simulations, we compare CIFPTA-S2
with our previous proposed IFPTA-S2 (iterative FP thresholding algorithm-Scheme 2 \cite{Li2019}) and a state-of-art method (Half algorithm \cite{Xu2012}). These numerical simulations 
are all conducted on a personal computer (3.40GHz, 16.0GB RAM) with MATLAB R2015b.

We generate a Gaussian random matrix of size $128\times 512$ with entries i.i.d. to Gaussian distribution, $\mathcal{N}(0,1)$, as the measurement matrix $\mathbf{A}$. The original $r$-sparse
signal $\mathbf{\bar{x}}$ with dimension $512$ is generated by choosing the non-zero locations over the support in random, and each nonzero entry is generated as follows\cite{Becker2011}:
\begin{equation}\label{equ46}
\mathbf{\bar{x}}_{i}=\eta_{1}[i]10^{\alpha\eta_{2}[i]},
\end{equation}
where $\eta_{1}[i]=\pm1$ with probability $1/2$ (a random sign), $\eta_{2}[i]$ is uniformly distributed in $[0,1]$ and the parameter $\alpha$ quantifies the dynamic range. We
generate the measurement vector $\mathbf{b}$ with dimension $128$ by $\mathbf{b}=\mathbf{A}\mathbf{\bar{x}}$, and therefore we know the sparsest solution to $\mathbf{A}\mathbf{\bar{x}}=\mathbf{b}$.
The stopping criterion is defined as
$$\frac{\|\mathbf{x}^{k+1}-\mathbf{x}^{k}\|_{2}}{\max\{\|\mathbf{x}^{k}\|_{2},1\}}\leq 10^{-15}$$
or maximum iteration step equal to 3000. The success recovery of the original sparse vector $\mathbf{\bar{x}}$ is measured by computing
$$\mathrm{RE}=\|\mathbf{x}^{\ast}-\mathbf{\bar{x}}\|_{2}.$$
In our numerical simulations, if $\mathrm{RE}\leq10^{-4}$, we say that the algorithm can exact recovery the original $r$-sparse signal $\mathbf{\bar{x}}$. For each simulation, we repeatedly perform
30 tests and present average results in this paper. In these numerical experiments, we set $\mu=0.99\|\mathbf{A}\|_{2}^{-2}$, $\zeta=10^{-4}$, $\tau=0.5$ and $c=0.5$ in our CIFPTA-S2. Moreover, 
we choose the parameter $\hat{a}$ as a random integer between 1 and 100 in CIFPTA-S2, and generated by the Matlab code: \texttt{randint(1,1,[1 100])}. In theses numerical simulations,
we set $a=2$, $\epsilon=0.01$ in our previous proposed IFPTA-S2. 

\begin{figure}[h!]
 \centering
 \includegraphics[width=0.45\textwidth]{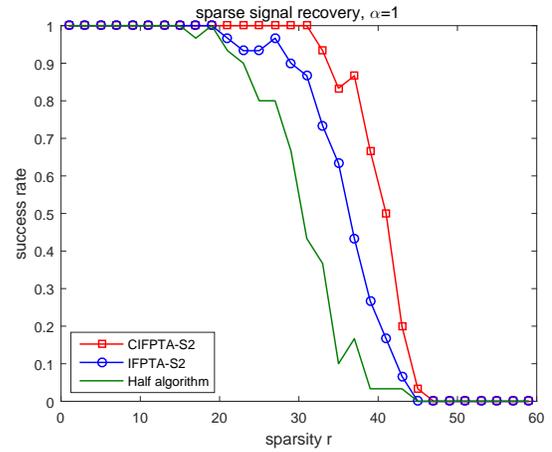}
\caption{The comparison of CIFPTA-S2, IFPTA-S2 and Half algorithm in the recovery of a sparse signal $\mathbf{\bar{x}}$ with different sparsity $r$, $\alpha=1$.}
\label{fig:3}
\end{figure}

The graphs demonstrated in Figs.\ref{fig:3}, \ref{fig:4} and \ref{fig:5} show the comparisons of CIFPTA-S2, IFPTA-S2 and Half algorithm in  recovering a sparse signal $\mathbf{\bar{x}}$ 
with different sparsity $r$. In these three numerical simulations, we set $\alpha=1$, $1.5$ and $2$ respectively. As we can see, the CIFPTA-S2 has the best performance compared with 
other two algorithms.

\begin{figure}[h!]
 \centering
 \includegraphics[width=0.45\textwidth]{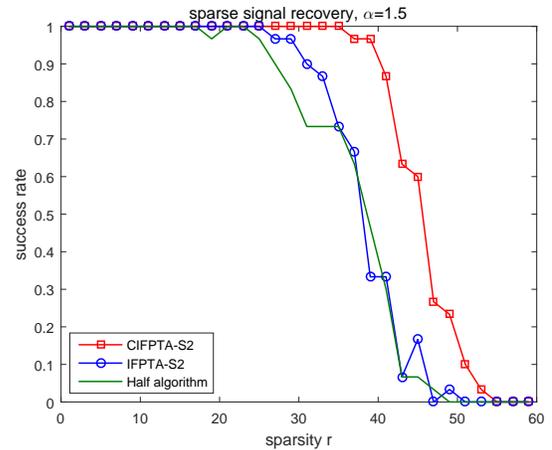}
\caption{The comparison of CIFPTA-S2, IFPTA-S2 and Half algorithm in the recovery of a sparse signal $\mathbf{\bar{x}}$ with different sparsity $r$, $\alpha=1.5$.}
\label{fig:4}
\end{figure}

\begin{figure}[h!]
 \centering
 \includegraphics[width=0.45\textwidth]{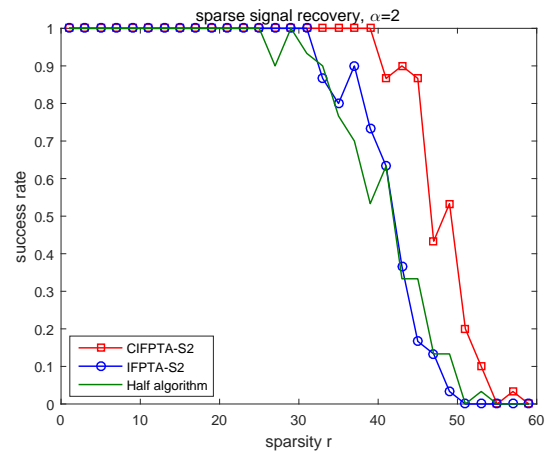}
\caption{The comparison of CIFPTA-S2, IFPTA-S2 and Half algorithm in the recovery of a sparse signal $\mathbf{\bar{x}}$ with different sparsity $r$, $\alpha=2$.}
\label{fig:5}
\end{figure}

\section{Conclusion}\label{section5}

The problem of recovering a sparse signal from the linear constraints, known as the $\ell_{0}$-norm minimization problem, has been attracting extensive attention in recent years. Unfortunately,
the $\ell_{0}$-norm minimization problem is a NP-hard problem. In this paper, we first review some known results from our latest work for our previous proposed iterative FP thresholding algorithm
to solve the problem $(FP^{\lambda}_{a})$, and then generate a convex iterative FP thresholding algorithm to solve the problem $(FP^{\lambda}_{a})$. Two parameter-setting strategies are given to
the choice of regularized parameter $\lambda$ and parameter $a$. Corresponding, two schemes of convex iterative FP thresholding algorithm are proposed to solve the problem $(FP^{\lambda}_{a})$.
One is convex iterative FP thresholding algorithm-Scheme 1 and the other is convex iterative FP thresholding algorithm-Scheme 2. A global convergence theorem is proved for the convex iterative FP
thresholding algorithm-Scheme 1. Under an adaptive rule for the choice of the regularized parameter $\lambda$ and parameter $a$, the convex iterative FP thresholding algorithm-Scheme 2 is adaptive
both for the choice of the regularized parameter $\lambda$ and parameter $a$. These are the advantage for our convex iterative FP thresholding algorithm-Scheme 2 compared with our previous proposed
two schemes of iterative FP thresholding algorithm. Numerical experiments on some sparse signal recovery problems have shown that our our convex iterative FP thresholding algorithm-Scheme 2 performs
the best in recovering a sparse signal compared with the iterative FP thresholding algorithm-Scheme 2 and Half algorithm.

\ifCLASSOPTIONcaptionsoff
  \newpage
\fi

\end{document}